\documentclass{amsart}  
\usepackage{amssymb}

\newtheorem{thm}{Theorem}
\newtheorem{lem}{Lemma}
\newtheorem{cor}{Corollary}
\newtheorem{prop}{Proposition}

\theoremstyle{definition}

\newtheorem{defn}{Definition}
\newtheorem{rem}{Remark}
\newtheorem{exmp}{Example}

\newcommand{\Cur}{\mathop{\fam 0 Cur}\nolimits}
\newcommand{\Cend}{\mathop{\fam 0 Cend}\nolimits}
\newcommand{\End}{\mathop{\fam 0 End}\nolimits}
\newcommand{\Hom}{\mathop{\fam 0 Hom}\nolimits}
\newcommand{\Aut}{\mathop{\fam 0 Aut}\nolimits}

\newcommand{\oo}[1]{\mathbin{{}_{(#1)}}}

\begin{document}

\title{Graded associative conformal algebras of finite type}

\author{Pavel Kolesnikov}

\address{Sobolev Institute of Mathematics, Novosibirsk, Russia}
\email{pavelsk@math.nsc.ru}

\begin{abstract}
In this paper, we consider graded associative conformal algebras.
The class of these objects includes pseudo-algebras over
non-cocommutative Hopf algebras of regular functions on
some linear algebraic groups.
In particular, an associative conformal algebra
which is graded by a finite group $\Gamma $
is a pseudo-algebra over the coordinate Hopf algebra of a
linear algebraic group $G$ such that the identity component $G^0$ is the affine line and
$G/G^0\simeq \Gamma $.
A classification of simple and semisimple graded associative conformal
algebras of finite type is obtained.
\end{abstract}

\maketitle

\section{Introduction}

Conformal algebras
were introduced in \cite{Kac1998} as a useful tool for studying vertex algebras
appeared in two-dimensional conformal field theory in mathematical physics
 \cite{BPZ}.
The structure of a (Lie) conformal algebra
encodes the singular part of the
operator product expansion (OPE) which is responsible for the
commutator of two chiral fields.

From the algebraic point of view, the notions of
conformal algebras, their representations and cohomologies
are higher-level analogues of the ordinary notions
in the pseudo-tensor category \cite{BD2004}
associated with the polynomial Hopf algebra $\Bbbk[D]$, see \cite{BDK2001} for
details.

For an arbitrary Hopf algebra $H$, an algebra in the pseudo-tensor
category associated with $H$ is called {\em pseudo-algebra}.
Note that ordinary algebras (representations, cohomologies)
correspond to the ``trivial'' case of 1-dimensional Hopf algebra.
For pseudo-algebras, the analogues of finite-dimensional
algebras are those finitely generated as modules over~$H$.
These pseudo-algebras are said to be {\em of finite type\/} or simply
{\em finite}.

The structure theory of finite Lie conformal algebras was developed
in \cite{DK1998} and later generalized in \cite{BDK2001} for pseudo-algebras
over a wide class of cocommutative Hopf algebras.
The structure theorems for finite associative conformal 
algebras and pseudo-algebras over cocommutative Hopf algebras
can be derived from the corresponding statements on Lie algebras.

Our aim is to obtain structure theorems for finite associative pseudo-algebras
over some non-cocommutative Hopf algebras (Lie pseudo-algebras can not be defined
in this case).

In \cite{Kol2008(TG)}, we proposed the notion of a conformal algebra
over a linear algebraic group $G$ which is equivalent to the notion
of a pseudo-algebra over the Hopf algebra $\Bbbk[G]$ of regular functions on~$G$.
Ordinary algebras over a field $\Bbbk $ correspond to the case of
trivial group, conformal algebras---to the case when $G$ is isomorphic to the
affine line $\mathbb A^1 = (\Bbbk, +)$.

In this paper, we introduce the notion of a graded conformal algebra over a
(connected) abelian linear algebraic group $G^0$.
In contrast to the case of ordinary algebras, we need three parameters
to define the grading: a group $\Gamma $, a homomorphism $\sigma :\Gamma \to \Aut G^0$,
and a 2-cocycle $\varphi $ of the group $\Gamma $ with coefficients in
$G^0$ endowed with $\Gamma$-module structure $\gamma\lambda = \lambda^{\sigma(\gamma)}$,
$\gamma \in \Gamma $, $\lambda \in G^0$.
The main reason for this approach is the following:
If $|\Gamma |<\infty$ then a $(\Gamma, \sigma, \varphi)$-graded
conformal algebra over $G^0$ is the same as
conformal algebra over a linear algebraic group $G$
obtained as the extension of $G^0$ by means of $\Gamma $
with respect to $\sigma $ and~$\varphi $.
However, the group $\Gamma $ may be infinite. In this case, a
$(\Gamma, \sigma, \varphi)$-graded conformal algebra is
no more a pseudo-algebra in the sense of \cite{BDK2001}.

The main result of the paper is the classification of finite simple and semisimple
associative $(\Gamma, \sigma, \varphi)$-graded conformal algebras over
$\mathbb A^1$.

\section{Preliminaries in conformal algebras and pseudo-algebras}

Suppose $A$ is an algebra (not necessarily associative) over a field 
$\Bbbk $, $\mathrm{char}\,\Bbbk=0$. 
Denote by $A[[z,z^{-1}]]$ the space of formal power series (distributions)
that may be infinite in both directions. A pair of distributions
$a(z), b(z)\in A[[z,z^{-1}]]$
is called {\em local\/} if there exists a natural $N$ such that
$a(w)b(z)(w-z)^N =0$ in the space $A[[z,z^{-1},w,w^{-1}]]$. 
The product $a(w)b(z)$ of two formal distributions that form a local pair
may be presented as 
\[
a(w)b(z) =
 \sum\limits_{n= 0}^{N-1} \dfrac{1}{n!} c_n(z)
 \dfrac{\partial^n\delta(w-z)}{\partial z^n}, 
\]
where $\delta(w-z)=\sum\limits_{m\in \mathbb Z} w^m z^{-m-1}$
is the formal delta-function, $c_n(z)\in A[[z,z^{-1}]]$ \cite{Kac1998}.

For example, 
if $V$ is a vertex operator algebra (see \cite{FLM1998} as a general reference)
with a state-field correspondence $Y:V\to \mathrm{gl}\, V[[z,z^{-1}]]$, 
$Y:a\mapsto Y(a,z)$,
then the space of all distributions $\{V(a,z)\mid a\in V\}$
consists of pairwise mutually local series.

For a local pair of distributions $a(z), b(z)\in A[[z,z^{-1}]]$
one may consider the Fourier transform of their product
\begin{equation}\label{eq:FourierTr}
(a\oo{\lambda}b)(z) =\mathrm{Res}_{w=0} a(w)b(z)\exp\{\lambda (w-z) \},
\end{equation}
where $\mathrm{Res}_{w=0}f(w,z,\lambda)$ is the formal residue of $f$, i.e., 
the coefficient at $w^{-1}$.
Locality implies the {\em $\lambda $-product\/} $(a\oo{\lambda } b)$ 
to be a polynomial in $\lambda $ with coefficients in $A[[z,z^{-1}]]$. 
Axiomatic description of the properties of this new ``parametrized'' operation
and its relationship with the ordinary derivation $T=d/dz$ on $A[[z,z^{-1}]]$
leads to the following

\begin{defn}[\cite{Kac1998}]\label{defn:Conformal}
A {\em conformal algebra\/} is a linear space $C$ over a field $\Bbbk $
of zero characteristic endowed with a linear map $T:C\to C$ 
and bilinear operation 
$(\cdot\oo{\lambda }\cdot): C\otimes C\to C[\lambda ]$
such that 
$(Ta\oo{\lambda }b)=-\lambda (a\oo{\lambda }b)$
and $(a\oo{\lambda }Tb)=(T+\lambda)(a\oo{\lambda }b)$.
\end{defn}

Conformal algebra which is finitely generated as $\Bbbk[T]$-module 
is said to be {\em finite}.

A subspace $C\subseteq A[[z,z^{-1}]]$ that consists of 
pairwise mutually local series which is closed under 
$d/dz$ and $(\cdot\oo{\lambda}\cdot)$ given by 
\eqref{eq:FourierTr} is called a conformal algebra 
of formal distributions over~$A$.

For every conformal algebra $C$ there exists a unique (up to 
isomorphism) {\em coefficient\/} algebra $\mathcal A(C)$  \cite{Roitman1999}
such that $C$ is embedded into 
a conformal algebra of formal distributions over $\mathcal A(C)$ 
and $\mathcal A(C)$ is universal among all such algebras $A$.
The following natural definition was proposed in \cite{Roitman1999}:
Given a variety $\mathfrak M$ of algebras
(associative, commutative, Lie, etc.), a conformal algebra 
$C$ is said to be $\mathfrak M$-conformal algebra 
if and only if $\mathcal A(C)$
belongs to $\mathfrak M$.
In particular, every vertex operator algebra gives rise to a Lie conformal algebra.
This is the main motivation to study Lie conformal algebras and their representations.

Structure theory of finite Lie and associative conformal algebras was 
developed in \cite{DK1998}. In particular, if $\Bbbk $ is an algebraically closed 
field of zero characteristic then every simple finite 
associative conformal algebra $C$
is isomorphic to the {\em current conformal algebra\/} over
$A=M_n(\Bbbk )$, i.e., 
\[
C\simeq \Bbbk[T]\otimes A, \quad 
(f(T)\otimes a)\oo{\lambda } (g(T)\otimes b) = f(-\lambda)g(T+\lambda )\otimes ab,
\]
for $f,g\in \Bbbk[T]$, $a,b\in A$.
Conformal algebra is called semisimple if it has no nonzero ideals $I$ such 
that $I\oo{\lambda } I=0$. A semisimple finite associative conformal algebra 
is proved to be a direct sum of simple ones \cite{DK1998}.

Note that for finite Lie conformal algebras there exists a unique exceptional 
example apart from current conformal algebras---the Virasoro algebra. 
This is a one-dimensional $\Bbbk[T]$ module with generator $v$ such that 
$v\oo{\lambda } v = (T+2\lambda )v$. 

An important generalization of conformal algebras was proposed in \cite{BDK2001}. 
For every Hopf algebra $H$ with coproduct $\Delta:H\to H\otimes H$
and antipode $S:H\to H$ one may consider the class 
$\mathcal M(H)$ of left unital $H$-modules 
as a pseudo-tensor category in the sense of \cite{BD2004}.
The notion of a pseudo-tensor category generalizes the notion of an operad. 
Namely, operads are exactly the pseudo-tensor category with only one object.
For one-dimensional 
Hopf algebra $H=\Bbbk$, $\mathcal M(H)$ is the class of linear spaces 
equipped by polylinear maps.

An algebra in $\mathcal M(H)$ is called a {\em pseudo-algebra over $H$}.
This is a module $C\in \mathcal M(H)$ endowed with a map 
\begin{equation}\label{eq:PseudoProduct}
*: C\otimes C \to (H\otimes H)\otimes _H C
\end{equation}
such that 
\[
(ha*gb)=((h\otimes g)\otimes_H 1)(a*b), \quad a,b\in C,\ h,g\in H.
\] 
Here $H\otimes H$ is considered as the outer product of regular right 
$H$-modules, i.e., $(h\otimes g)\cdot f = (h\otimes g)\Delta(f)$,
$f,g,h\in H$.

Conformal algebras in the sense of Definition \ref{defn:Conformal}
are exactly pseudo-algebras over $H=\Bbbk[T]$, $\Delta(T)=T\otimes 1+1\otimes T$,
$S(T)=-T$. Namely, if $C$ is a conformal algebra, $a,b\in C$,
$(a\oo{\lambda} b)= \sum\limits_{n\ge 0}\lambda ^n c_n$
then 
$a*b = \sum\limits_{n\ge 0} (-1)^n(T^n\otimes 1)\otimes_H c_n$.

From the categorical point of view (similar to \cite{GK1994}), 
a pseudo-algebra is a functor from the operad of planar binary 
trees to $\mathcal M(H)$. 
Given a variety $\mathfrak M$ of (ordinary) algebras defined by 
a family of homogeneous polylinear identities, 
one may define what is an $\mathfrak M$-pseudo-algebra 
over $H$ by means of the operad associated with $\mathfrak M$.
For conformal algebras, this definition is equivalent to the mentioned above:
An $\mathfrak M$-pseudo-algebra over $\Bbbk[T]$ ($\mathrm{char}\,\Bbbk =0$)
is the same as $\mathfrak M$-conformal algebra \cite{Kol2006CA}.

A systematic study of Lie and associative pseudo-algebras over 
a wide class of cocommutative Hopf algebras was done in \cite{BDK2001}.

The operad associated with Lie 
(or commutative) algebras is symmetric, but $\mathcal M(H)$
is a symmetric pseudo-tensor category only if $H$ is cocommutative. 
This is why one cannot define Lie pseudo-algebras over a non-cocommutative
Hopf algebras in this way. 
Nevertheless, associative pseudoalgebras may be defined over 
an arbitrary Hopf algebra.

A natural class of Hopf algebras is given by coordinate algebras 
of linear algebraic groups. 
If $G$ is such a group and $H=\Bbbk[G]$ is the algebra of regular 
functions on $G$ then the class of pseudo-algebras over $H$ 
can be completely described in terms of operations indexed by 
elements of~$G$.

\begin{defn}[\cite{Kol2008(TG)}]
A {\em conformal algebra  over $G$} is a left unital 
$H$-module $C$ endowed with a family of bilinear operations 
\[
(\cdot\oo{g}\cdot): C\otimes C\to C, \quad g\in G,
\]
such that the following properties hold.
\begin{itemize}
\item 
For every $a,b\in C$ there exists $\sum\limits_i h_i\otimes c_i\in H\otimes C$
such that $(a\oo{g} b) = \sum\limits_ih_i(g)c_i$ for all $g\in G$; 
\item
$(ha\oo{g} b) = h(g^{-1})(a\oo{g} b)$;
\item
$(a\oo{g} hb) = L_gh (a\oo{g} b)$, where $(L_gh)(x)=h(gx)$ for $x\in G$.
\end{itemize}
\end{defn}

Obviously, if 
$G=\mathbb A^1\simeq (\Bbbk, +)$ then this definition 
coincides with Definition \ref{defn:Conformal}. In general, 
a conformal algebra over $G$ is the same as pseudo-algebra over $H=\Bbbk[G]$.
Associativity of such a pseudo-algebra $C$ may be expressed in terms of 
$g$-products as follows:
\[
a\oo{g}(b\oo{h} c) = (a\oo{g}b)\oo{hg} c, \quad a,b,c\in C,\ g,h\in G.
\]

The main purpose of this paper is to describe the structure of 
simple and semisimple conformal algebras over linear algebraic 
groups $G$ such that $G^0$ is isomorphic to the affine line.
It turns out that such objects are natural to describe in terms 
of graded conformal algebras over the affine line. 
This is a step towards complete structure theory of conformal algebras 
over 1-dimensional linear algebraic groups. 
Another case (when $G^0\simeq \Bbbk^*$) corresponds 
to $\mathbb Z$-conformal algebras introduced in \cite{GK1997}. 
Structure theory of such algebras is a topic for further investigation.

\section{Graded conformal algebras}

Let $G^0$ be an abelian connected linear algebraic group, and let
$\Aut G^0$ be the group of its continuous automorphisms.
Suppose $\Gamma $ is a group, and fix a homomorphism $\sigma : \Gamma \to \Aut G^0$.
Also, choose a 2-cocycle
$\varphi \in Z^2(\Gamma, G^0, \sigma)$, i.e., a map
$\varphi: \Gamma\times \Gamma \to G^0$ such that
\begin{equation}\label{eq:CocycleCondition}
 \varphi(\alpha\beta, \gamma)\varphi(\alpha, \beta )^{\sigma(\gamma)}
 = \varphi(\alpha, \beta\gamma)\varphi(\beta, \gamma)
\end{equation}
for all $\alpha , \beta, \gamma \in \Gamma$ and $\varphi(e,e)=e$.
Note that $\varphi(\alpha, e)=\varphi(e,\alpha)=e$ and
$\varphi(\alpha^{-1},\alpha)=\varphi(\alpha,\alpha^{-1})^{\sigma(\alpha)}$
for all $\alpha \in \Gamma $.

Denote by $H^0$ the algebra of regular functions on $G^0$.
For $\lambda \in G^0$, denote by $L_\lambda $ the shift operator on $H^0$,
i.e., $(L_\lambda h)(\mu )= h(\lambda \mu)$, $h\in H^0$, $\mu \in G^0$.

\begin{defn}\label{defn:GradedConformalG0}
A $(\Gamma, \sigma, \varphi)$-{\em graded conformal algebra over\/} $G^0$
is a graded $H^0$-module
\[
C = \bigoplus\limits_{\gamma \in \Gamma }C_\gamma
\]
equipped by a family of $\Bbbk $-bilinear operations
$(\cdot \oo{\lambda }\cdot)$, $\lambda \in G^0$, such that:
\begin{itemize}
\item[(C1)]
For every $a,b\in C$ the function
$(a\oo{x} b)$ is regular, i.e.,
$(a\oo{\lambda } b)=\sum_i h_i(\lambda )c_i$
for an appropriate $\sum_i h_i\otimes c_i \in H^0\otimes C$.
\item[(C2)]
If $a\in C_\alpha $ and $b\in C_\beta $ then $(a\oo{\lambda } b)\in C_{\alpha\beta}$,
$\alpha,\beta \in \Gamma$, $\lambda \in G^0$.
\item[(C3)]
If $a\in C_\alpha $, $b\in C_\beta $, and $h\in H^0$ then
\[
\begin{gathered}
(ha\oo{\lambda }b)= h((\lambda^{-1}\varphi(\alpha, \alpha^{-1})^{-1})^{\sigma(\alpha)})
  (a\oo{\lambda }b), \\
(a\oo{\lambda }hb)=
(L_{\lambda^{\sigma(\alpha\beta)}\varphi(\alpha^{-1},\alpha\beta)}h) (a\oo{\lambda } b)
\end{gathered}
\]
for all $\lambda \in G^0$.
\end{itemize}
\end{defn}

If $C$ is a finitely generated $H^0$-module then $C$ is said to be {\em finite}
$(\Gamma, \sigma, \varphi)$-graded conformal algebra over $G^0$.

A homomorphism of $(\Gamma, \sigma, \varphi)$-graded conformal algebras over
$G^0$ is a homogeneous homomorphism of graded $H^0$-modules
preserving the operations $(\cdot\oo{\lambda }\cdot)$
for all $\lambda \in G^0$.

Recall that 2-cocycles $\varphi, \varphi' \in Z^2(\Gamma, G^0, \sigma)$
are called cohomological if there exists a 1-cochain $\tau :\Gamma \to G^0$
such that
\begin{equation}\label{eq:Coboundary}
\varphi'(\alpha, \beta)\varphi(\alpha,\beta)^{-1}
 = \tau(\alpha)^{\sigma(\beta)}\tau(\beta)\tau(\alpha\beta)^{-1},
\quad \alpha,\beta\in \Gamma .
\end{equation}
Given $\Gamma $ and $\sigma $ as above, a 2-cocycle
 $\varphi\in Z^2(\Gamma, G^0,\sigma)$
defines an extension
\[
\{e\}\to G^0\to G \to \Gamma \to \{e\},
\]
where $\sigma $ represents the conjugation:
$g^{-1}\lambda g = \lambda ^{\sigma(\gamma)}$, where
 $\gamma =gG^0\in \Gamma=G/G^0$.
Cocycles $\varphi $ and $\varphi'$ define isomorphic extensions if and only if
they are cohomological (see, e.g., \cite{Brown}).

\begin{prop}
If $\varphi,\varphi'\in Z^2(\Gamma, G^0,\sigma)$ are cohomological then
the categories of $(\Gamma, \sigma, \varphi)$- and
$(\Gamma, \sigma, \varphi')$-graded conformal algebras over $G^0$ are
equivalent.
\end{prop}

\begin{proof}
Suppose $C=\bigoplus\limits_{\gamma\in \Gamma }C_\gamma $
is a $(\Gamma, \sigma, \varphi)$-graded conformal algebra, and
$\varphi'\in Z^2(\Gamma, G^0,\varphi)$ is related with $\varphi $
by  \eqref{eq:Coboundary}.
Then we may define a new structure of an $H^0$-module and a new family
of operations $(\cdot{}_{[\lambda ]}\cdot)$ on $C$ by the following rule:
\[
h\cdot a = (L_{\tau(\alpha)^{-1}}h)a,
\quad
(a{}_{[\lambda]}b)= (a\oo{\lambda\tau(\alpha^{-1})}b)
\]
for $a\in C_\alpha$, $b\in C$, $\lambda \in G^0$, $h\in H^0$.
Straightforward computations show this new system is
a $(\Gamma, \sigma ,\varphi')$-graded conformal algebra over $G^0$
(it is enough to check the condition (C3) of Definition~\ref{defn:GradedConformalG0}).
The correspondence constructed is obviously an equivalence of categories.
\end{proof}

\begin{exmp}
If $G^0$ is the trivial group then $\sigma $ and $\varphi$ are also trivial,
so Definition~\ref{defn:GradedConformalG0} describes the class of
$\Gamma $-graded ordinary algebras over~$\Bbbk $.
\end{exmp}

\begin{exmp}
Let $A=\bigoplus\limits_{\gamma\in \Gamma }A_\gamma $
be an (ordinary) $\Gamma $-graded algebra. Then the free
$H^0$-module $C=H^0\otimes A$ can be endowed with a natural structure
of a $(\Gamma ,\sigma,\varphi)$-graded conformal algebra over $G^0$
via
\[
(h\otimes a)\oo{\lambda }(g\otimes b) =
h((\lambda^{-1}\varphi(\alpha, \alpha^{-1})^{-1})^{\sigma(\alpha)})
(L_{\lambda^{\sigma(\alpha\beta)}\varphi(\alpha^{-1},\alpha\beta)}g)\otimes ab,
\]
$h,g\in H^0$, $a\in A_\alpha$, $b\in A_\beta $.
The graded conformal algebra $C$ obtained is called the
{\em current conformal algebra\/} over $A$, it is denoted by $\Cur A$.
\end{exmp}

The following example explains the origin of Definition \ref{defn:GradedConformalG0}.

\begin{exmp}
Consider a linear algebraic group $G$
with abelian identity component $G^0$. Denote
by $\Gamma $ the quotient group $G/G^0$.
Then there exist $\sigma: \Gamma \to \Aut G^0$ and
$\varphi \in Z^2(\Gamma , G^0, \sigma )$ such that
$G$ is isomorphic to the set $\Gamma \times G^0$ endowed with
the product
\[
(\gamma,\lambda)(\beta,\mu)
 = (\gamma\beta,  \lambda^{\sigma(\beta)} \mu\varphi(\gamma, \beta) ),
\]
$\gamma, \beta \in \Gamma$, $\lambda, \mu \in G^0$.

The Hopf algebra $H=\Bbbk[G]$
 of regular functions on $G$
may be presented as the algebra
 $(\Bbbk \Gamma)^*\otimes H^0$
with coproduct and antipode depending on $\sigma $ and $\varphi $.
Here $(\Bbbk\Gamma)^*$ is the dual group algebra of $\Gamma $
spanned by functionals $T_\gamma: \Gamma \to \Bbbk $,
$\langle T_\gamma, \beta \rangle =\delta_{\gamma,\beta}$,
$\gamma,\beta \in \Gamma $.

Suppose $C$ is a conformal algebra over $G$ \cite{Kol2008(TG)},
i.e., a pseudo-algebra over $H$ \cite{BDK2001}.
The natural decomposition
$C=\bigoplus\limits_{\gamma \in \Gamma } C_\gamma$,
$C_\gamma =\{x\in C\mid (T_\gamma\otimes 1) x=x \}$,
defines a $\Gamma $-grading on the $H^0$-module~$C$.
Moreover, the operations
\[
(a\oo{\lambda }b) = (a\oo{\alpha^{-1},\lambda} b),
\quad a\in C_\alpha,\  b\in C,\  \lambda \in G^0,
\]
turn $C$ into a $(\Gamma, \sigma, \varphi)$-graded conformal algebra over
$G^0$.

Conversely, every $(\Gamma, \sigma, \varphi)$-graded conformal algebra
over $G^0$ (for a finite group $\Gamma $) is a conformal algebra
over the corresponding extension of $G^0$.
\end{exmp}

Let $V=\bigoplus\limits_{\gamma \in \Gamma } V_\gamma $ be a
graded linear space. Then the algebra $\End V$ of linear
transformations of $V$ carries the following natural grading:
\begin{gather}\label{eq:NatGrading}
\End V = \bigoplus\limits_{\gamma \in \Gamma } (\End V)_\gamma, \\
(\End V)_\gamma = \{\xi\in \End V \mid
 \xi(V_\alpha)\subseteq V_{\gamma\alpha},\ \alpha\in \Gamma \}.
\end{gather}

A map $a:G^0\to \End V$ is called locally regular if for every $v\in V$
there exists an element $(a\oo{x} v)=\sum\limits_i h_i\otimes v_i\in H^0\otimes V$
such that $a(\lambda )v = \sum\limits_i f_i(\lambda )v_i$ for all $\lambda \in G^0$.

\begin{exmp}
Let $M=\bigoplus\limits_{\gamma \in \Gamma } M_\gamma $ be a
graded $H^0$-module.
Denote by $(\Cend^\Gamma M)_\gamma $ the space of all locally regular
maps $a: G^0\to (\End M)_\gamma $ such that
\[
a(\lambda )(hu)=
 (L_{\lambda^{\sigma(\gamma\beta)}\varphi(\gamma^{-1}, \gamma\beta)}h)a(\lambda )u
\]
for all $h\in H^0$, $u\in M_\beta$, $\lambda \in G^0$.

In fact, the structure of $(\Cend^\Gamma M)_\gamma $ depends also
on $\sigma $ and $\varphi $, but we use the single letter $\Gamma $
in order to simplify notations. We will later use $\Cend $ without a
superscript for the space of ``non-graded''
conformal linear maps, as in~\cite{Kac1998}.

For each $\gamma \in \Gamma$ the space $(\Cend^\Gamma  M)_\gamma $ carries the
following $H^0$-module structure:
\[
(ha):\lambda \to h((\lambda^{-1}\varphi(\gamma, \gamma^{-1})^{-1})^{\sigma(\gamma )}).
\]
Consider the
formal direct sum of all $(\Cend^\Gamma M)_\gamma $ as a
graded $H^0$-module denoted by $\Cend^\Gamma M$. There exists a natural
family of bilinear operations  $(\cdot\oo{\lambda }\cdot)$ on
$\Cend^\Gamma M$, namely,
\begin{equation}\label{eq:CendOperations}
(a\oo{\lambda }b)(\mu ) =
 a(\lambda )b((\lambda^{-1}\mu\varphi(\beta^{-1},\alpha^{-1})^{-1})^{\sigma(\alpha)}  )
\end{equation}
for $a\in (\Cend^\Gamma M)_\alpha$, $b\in (\Cend^\Gamma M)_\beta $,
$\lambda ,\mu \in G^0$.
It is easy to check that the operations given by
\eqref{eq:CendOperations} satisfy the conditions (C2) and (C3)
of Definition~\ref{defn:GradedConformalG0}.
Moreover, if $M$ is a finitely generated $H^0$-module then
the condition (C1) also holds. Indeed,
for a fixed system of homogeneous generators $u_1,\dots, u_n\in M$,
every map $a\in \Cend^\Gamma M$ is completely  defined by no more than $n^2$
 regular functions $h_{ij}\in H^0$.
Indeed, it is enough to consider
\[
a(x)u_i = \sum\limits_{j=1}^n f_{ij}\otimes u_j, \quad i=1,\dots, n.
\]
Hence, $(a\oo{x} b)$ can also be considered as an element in
$H^0\otimes \Cend^\Gamma M$.
\end{exmp}

For every $(\Gamma, \sigma,\varphi)$-graded conformal algebra $C$ over
$G^0$ there exists a natural homogeneous $H^0$-linear map
$\ell: C\to \Cend^\Gamma C$,
$\ell(a):\lambda \mapsto (a\oo{\lambda }\cdot) \in \End C$,
$a\in C$, $\lambda \in G^0$.
If $\ell$ preserves operations
$(\cdot\oo{\lambda}\cdot)$ for all $\lambda \in G^0$ then $C$ is
said to be {\em associative}. This definition is similar to the case of ordinary algebras
when $\Cend $ means the same as $\End $.
Thus, the axiom of associativity for graded conformal algebras has the
following form:
\begin{equation}\label{eq:Assoc}
(a\oo{\lambda }b)\oo{\mu }c =
 (a\oo{\lambda}
(b
\oo{(\lambda^{-1}\mu\varphi(\beta^{-1},\alpha^{-1})^{-1})^{\sigma(\alpha)} }c)),
\end{equation}
$a\in C_\alpha$, $b\in C_\beta$, $c\in C$, $\lambda, \mu \in G^0$.

In particular, if $A$ is an associative graded algebra then so is
$\Cur A$, if $M$ is a finitely generated graded $H^0$-module then
$\Cend^\Gamma M$ is an associative graded conformal algebra.

Let $M=\bigoplus\limits_{\gamma \in \Gamma} M_\gamma $
be a finitely generated graded $H^0$-module.
A representation $\rho $ of an associative $(\Gamma,\sigma, \varphi)$-graded
conformal algebra
$C$ over $G^0$ on $M$ is a homomorphism $\rho :C\to \Cend^\Gamma M$ of
$(\Gamma,\sigma,\varphi)$-graded conformal algebras over $G^0$.
Given a representation of $C$ on $M$, the latter is called a graded
conformal $C$-module.

As in the case of ordinary algebras, we use the following convention:
$\rho(a)(\lambda )v = a\oo{\lambda } v$, $a\in C$,
$v\in M$.
If $\rho $ is injective then the $C$-module $M$ is said to be faithful;
if $\rho \ne 0$ and there are
no nonzero proper graded $C$-invariant $H^0$-submodules in $M$
then $M$ is called irreducible.
Irreducibility of $M$ means that
for every $0\ne v\in M_\gamma $ and for every $u\in M_{\alpha\gamma}$
($\alpha, \gamma\in \Gamma $) there exist
$h_i\in H^0$, $a_i\in C_\alpha$, $\lambda _i\in G^0$, $i=1,\dots, n$,
such that
\[
u=\sum\limits_{i=1}^n h_i(a_i \oo{\lambda_i} v).
\]

\begin{defn}\label{defn:IrredSubalgebra}
A graded subalgebra $C$ of the $(\Gamma , \sigma, \varphi)$-graded
conformal algebra $\Cend^\Gamma M$ is irreducible if
$M$ is an irreducible graded conformal $C$-module.
\end{defn}

If $C\oo{\lambda }C\ne 0$ for some $\lambda \in G^0$ and
$C$ has no non-zero proper graded ideals
then $C$ is said to be simple.
A graded ideal $I$ of $C$ is abelian if $I\oo{\lambda }I=0$
for all $\lambda \in G^0$.
If $C$ has no nonzero abelian graded ideals then $C$ is called semisimple.

\begin{lem}\label{lem:TorsionAnnihilates}
Let $C$ be a $(\Gamma, \sigma, \varphi)$-graded associative
conformal algebra over $G^0$, $U$ be a graded conformal $C$-module,
and let $W\subseteq U$ be a graded $C$-submodule such that
$hU\subseteq W$ for some $h\in H^0$.
Then $(C\oo{\lambda } U) \subseteq W$ for all $\lambda \in G^0$.
\end{lem}

\begin{proof}
Consider an arbitrary pair
of homogeneous elements $a\in C_\alpha $,
$u\in U_\beta$, where $\alpha, \beta \in \Gamma $.
Suppose $(a\oo{x} u)=\sum_i h_i\otimes v_i\in H^0\otimes U$.
Then
\[
(a\oo{x} hu) = \sum\limits_{i, (h)} h_i s_{\alpha,\beta}(h_{(1)})\otimes h_{(2)}v_i,
\]
where $\sum\limits_{(h)} h_{(1)}\otimes h_{(2)}$ is the Sweedler's notation
for the standard coproduct in $H^0$,
$s_{\alpha,\beta }$ is the following automorphism of the algebra $H^0$:
\[
s_{\alpha,\beta }(f):\lambda
 \mapsto (L_{\varphi(\alpha^{-1},\alpha\beta)}f)(\lambda ^{\sigma(\alpha\beta)}),
\quad
f\in H^0,\ \lambda \in G^0.
\]

Consider the linear map $\Psi_{\alpha,\beta}:H^0\otimes U\to H^0\otimes U$
defined as follows:
\[
\Psi_{\alpha,\beta}(f\otimes v)
 = \sum\limits_{(f)} s_{\alpha,\beta}(f_{(1)})\otimes f_{(2)}v,
\quad v\in U.
\]
This map is invertible since
\[
\Psi^{-1}_{\alpha,\beta}(f\otimes v)
 = \sum\limits_{(s_{\alpha,\beta}^{-1}(f))}
 (s_{\alpha,\beta}^{-1}(f))_{(1)}\otimes S((s_{\alpha,\beta}^{-1}(f))_{(2)})v,
\]
where $S$ is the standard antipode on $H^0$.

Since $(a\oo{x} hu) \in H^0\otimes W \subset H^0\otimes U$, we also have
\begin{equation}\label{eq:FourierTransform}
\Psi^{-1}_{\alpha,\beta}(a\oo{x}hu)
=
(h\otimes 1)\Psi^{-1}_{\alpha,\beta}(a\oo{x} u) \in H^0\otimes W.
\end{equation}
The group $G^0$ is connected and $H^0$ has no zero divisors.
Therefore, \eqref{eq:FourierTransform} implies
$\Psi^{-1}_{\alpha,\beta}(a\oo{x} u)\in H^0\otimes W$.
By the definition of $\Psi_{\alpha,\beta}$,
$(a\oo{x} u) \in H^0\otimes W$.
\end{proof}

\begin{prop}\label{prop:CendEmbedding}
Let $C$ be a finite simple $(\Gamma,\sigma, \varphi)$-graded
conformal algebra over $G^0$.
Then there exists a finite faithful irreducible graded
conformal $C$-module $M$ which is a torsion-free $H^0$-module.
\end{prop}

\begin{proof}
Since $H^0$ is a Noetherian algebra, $C$ is a Noetherian
$H^0$-module. There exists a maximal graded left ideal $I$ of $C$.
Then $M=C/I$ is an irreducible graded conformal $C$-module which
is obviously faithful.

Moreover, $M$ is a torsion-free $H^0$-module.
Indeed, consider the graded left ideal
 $I_1=\{a\in C\mid a+I \in \mathrm{Tor}_{H^0}(C/I)\}$
which contains~$I$. Since $I$ is maximal, we have either
$I_1=I$ or $I_1=C$. The first option means
$\mathrm{Tor}_{H^0}(C/I)=0$, so assume the second holds.

In this case, there exists $h\in H^0\setminus \{0\}$
such that $hC\subseteq I$.
By Lemma \ref{lem:TorsionAnnihilates},
$(C\oo{\lambda } C)\subseteq I$
for all $\lambda \in G^0$, which is impossible for a simple algebra.
\end{proof}

\begin{prop}\label{prop:IrreducibleSemisimple}
Let $C$ be an irreducible graded subalgebra of
the $(\Gamma, \sigma, \varphi)$-graded conformal algebra $\Cend^\Gamma M$
(over $G^0$)
for a finitely generated graded $H^0$-module $M$.
Then the homogeneous component $C_e$
 is a semisimple conformal algebra over $G^0$.
\end{prop}

\begin{proof}
For an arbitrary
$0\ne v\in M_\gamma $, the irreducibility of $C$ implies
\[
C_e v := H^0\{a\oo{\lambda } v\mid a\in C_e, \, \lambda \in G^0\} =  M_\gamma .
\]
Suppose $I$ is an abelian ideal of $C_e$, $I\ne 0$.
If $0\ne I v\ne M_\gamma $ then $C_e (I v)\subseteq Iv \ne M_\gamma $
which contradicts the irreducibility of $M$. Therefore, for each
$\gamma \in \Gamma $ either $IM_\gamma =0$ or $IM_\gamma =M_\gamma $.
Since $M$ is faithful and $I\ne 0$, there exists at least one
 $\gamma \in \Gamma $
such that $M_\gamma \ne 0$ and $IM_\gamma =M_\gamma $.
But then $M_\gamma =I(IM_\gamma )=0  $, a contradiction.
\end{proof}

\section{Irreducible graded algebras of linear transformations}

Suppose $V=\bigoplus\limits_{\gamma \in \Gamma } V_\gamma $
is a finite-dimensional $\Gamma $-graded linear space.
Then $\End V$ is a $\Gamma $-graded associative algebra
with respect to the natural grading \eqref{eq:NatGrading}.
In this section, we consider the following problem: Describe
{\em graded irreducible} subalgebras of $\End V$,
i.e., graded subalgebras $A\subseteq \End V$ such that $V$ has no
nonzero proper graded $A$-invariant subspaces. (In the non-graded
settings, the answer is given by the classical Burnside Theorem.)
The case $\Gamma =\mathbb Z_2$ was considered in \cite{Ber1987}.
If $\Gamma $ is a finite group then this problem naturally fits within
the frames of conformal algebras over $\Gamma $; it was solved in
\cite{Kol2008(TG)}. Here we consider the general case.

Without loss of generality, we may assume $V_e\ne 0$. Otherwise, we may
shift the grading on $V$ by setting $V_\gamma'=V_{\gamma g}$ for a fixed
element $g\in \Gamma $. The corresponding grading on $\End V$ and
its graded irreducible subalgebras remain the same.

In the case $|\Gamma |<\infty$, this problem was solved in \cite{Kol2008(TG)}
in the context of conformal algebras over $\Gamma $.
If $\Gamma $ is an infinite group then $(\Bbbk\Gamma )^*$ is not a Hopf algebra,
but we may still obtain a very similar result.

Let $T_\gamma $, $\gamma \in \Gamma $, stands for the projection
$V\to V_\gamma $ relative to the given grading on~$V$.
Denote by $\Gamma _0$ the set of all $\gamma \in \Gamma $ such that
$V_\gamma =0$.
The identity operator $\sum\limits_{\gamma\in \Gamma }T_\gamma $
is actually presented by finite sum since almost all $T_\gamma $
are zero.

A subgroup $\Gamma_1\le \Gamma $ is said to be {\em fine\/} if
\begin{equation}\label{eq:Splitting}
\Gamma \setminus \Gamma_0=
\Gamma_1 \cup \Gamma_2 \cup
\dots \cup \Gamma_p ,\quad \Gamma_k = \gamma_k\Gamma _1.
\end{equation}
Let us fix the system of representatives $\gamma_1,\dots, \gamma_p\in \Gamma $
assuming $\gamma_1=e$.

Denote by $Z(\Gamma,\Gamma_0,\Gamma_1)$ the set of all maps
$\chi : \Gamma\times \Gamma \to \Bbbk$ such that:
\begin{gather}
\chi(\gamma, \beta) =0,\quad \beta\in \Gamma_0\ \mbox{or}\
  \gamma\beta\in \Gamma_0,\nonumber \\
\chi(\gamma, \alpha\gamma_k)\chi(\gamma\alpha, \beta) =
 \chi(\gamma, \alpha\beta)\chi(\alpha, \beta),
\quad \alpha,\gamma\in \Gamma,\ \beta\in \Gamma_k,
                                            \label{eq:ZCondition}
\end{gather}
and for $\gamma,\beta \notin \Gamma_0$ we have
\[
\begin{gathered}
\chi(\gamma, \beta)=1, \quad \beta=\gamma_k \ \mbox{or} \ \gamma =e, \\
\chi(\gamma, \beta)\ne 0\quad \mbox{otherwize}.
\end{gathered}
\]
For example, if $\Gamma_1=\Gamma $ then
$Z(\Gamma,\Gamma_0,\Gamma_1)=Z^2(\Gamma, \Bbbk^*)$. A restriction
of $\chi \in Z(\Gamma,\Gamma_0,\Gamma_1)$ to $\Gamma_1\times \Gamma _1$
is a 2-cocycle from $Z^2(\Gamma_1, \Bbbk^* )$. Moreover, for every
$\theta \in Z^2(\Gamma_1, \Bbbk^*)$ such that $\theta(e,e)=1$
the function
\[
\chi(\gamma, \beta) = \begin{cases}
\theta(\gamma_q^{-1}\gamma \gamma_k, \gamma_k^{-1}\beta), &
     \beta\in \Gamma_k, \ \gamma\beta\in \Gamma_q, \\
0, & otherwize,
\end{cases}
\]
 belongs to $Z(\Gamma,\Gamma_0,\Gamma_1)$.

\begin{thm}\label{thm:A_Structure}
Let $A$ be a graded irreducible subalgebra of $\End V$.
Then there exist:
\begin{itemize}
\item
A fine subgroup $\Gamma_1\le \Gamma $ with a family of $\gamma_k\in \Gamma $,
$k=1,\dots, p$, satisfying \eqref{eq:Splitting};
\item
A function $\chi\in Z(\Gamma,\Gamma_0,\Gamma_1)$;
\item
A family of linear isomorphisms
$\iota_\gamma:V_{\gamma_k}\to V_{\gamma}$, $\gamma \in \Gamma_k$, where
$\iota_{\gamma_k}=\mathrm{id}_{V_{\gamma_k}}$,
\end{itemize}
such that
\begin{equation}\label{eq:A-Structure}
 A_\alpha =
 \left\{\sum\limits_{k=1}^p \sum\limits_{\gamma\in \Gamma_k}
 \chi(\alpha, \gamma)\iota_{\alpha\gamma}\iota_{\alpha\gamma_k}^{-1}a_k
 \iota_\gamma^{-1} T_\gamma
\mid
a_k\in \mathrm{Hom}\,(V_{\gamma_k}, V_{\alpha\gamma_k})\right\},
\quad \alpha \in \Gamma.
\end{equation}
\end{thm}

The isomorphisms $\iota_{\alpha\gamma }$ and
$\iota_{\alpha\gamma_k}$ in the right-hand side of \eqref{eq:A-Structure}
are defined only when $\alpha\Gamma_k \not\subseteq \Gamma_0$.
Otherwise, $\chi(\alpha,\gamma )=0$, so we just do not need to
take into consideration those summands corresponding to such index~$k$.

\begin{proof}
Proposition \ref{prop:IrreducibleSemisimple} applied to $G^0=\{e\}$
implies the algebra $A_e$ is semisimple. By the Wedderburn
Theorem,
$A_e=I_1\oplus\dots \oplus I_p$, where
$I_k\simeq M_{n_k}(\Bbbk )$.
For every $a\in A_e$ we have
\begin{equation}\label{eq:E-decomp}
a \left(\sum\limits_{\beta\in \Gamma} T_\beta\right)
=
\sum\limits_{\gamma\in \Gamma\setminus\Gamma_0}\pi_\gamma(a),
\end{equation}
where the map $\pi_\gamma: A_e\to \End V_\gamma $
is defined by $\pi_\gamma(a) = aT_\gamma $.
This is a homomorphism of algebras.

For each $k=1,\dots, p$ set
$\Gamma_k=\{\gamma\in \Gamma \mid
 \pi_\gamma (I_k) \ne 0\}\ne \emptyset $.
Obviously, $\Gamma \setminus \Gamma_0
=\Gamma_1\cup \dots \cup\Gamma_p$,
$\Gamma_i\cap \Gamma_j =\emptyset$ for
$i\ne j$.

Let us fix a system of representatives $\gamma_k\in \Gamma_k$
for $k=1,\dots, p$, and assume $e\in \Gamma_1$,
$\gamma_1=e$.
If $\gamma\in \Gamma_k$
then $(\dim V_\gamma)^2=\dim I_k$, so
$n_\gamma = n_k$.
Moreover, $\pi_\gamma (I_k) = \End V_\gamma $.
Hence,
$\pi_\gamma \pi^{-1}_{\gamma_k} :\End V_{\gamma_k} \to \End V_\gamma$.
is also an isomorphism,
so it can be presented by conjugation:
There exist linear isomorphisms
$\iota_\gamma: V_{\gamma_k}\to V_\gamma $,
$\gamma\in \Gamma_k$, $k=1,\dots, p$,
such that $\pi_\gamma\pi_{\gamma_k}^{-1}(a)
= \iota_\gamma a \iota_\gamma^{-1}$
for $a\in \End V_{\gamma_k}$.

Finally, \eqref{eq:E-decomp} implies
\begin{equation}\label{eq:Ae-Structure}
A_e = \left\{
\sum\limits_{k=1}^p\sum\limits_{\gamma \in \Gamma_k}
\iota_\gamma a_k \iota_\gamma^{-1} T_\gamma
\mid
a_k\in \End V_{\gamma_k}
\right\}.
\end{equation}

\begin{lem}
The subsets $\Gamma_k\subseteq \Gamma $,
$k=1,\dots, p$, are left adjacent classes
in $\Gamma $ with respect to finite subgroup
$\Gamma_1$.
\end{lem}

\begin{proof}
It is enough to show the following property:
$\gamma \Gamma_k\cap \Gamma _m\ne \emptyset $
implies $\Gamma_m \subseteq \gamma\Gamma_k$
for every $\gamma \in \Gamma $ and for every
$k,m\in \{1,\dots, p\}$.
Indeed, the same property for $\gamma^{-1}$
would show $\gamma \Gamma_k$ is equal to $\Gamma_m$
provided they have nonempty intersection.
Hence, the finite set $\Gamma_1$ is a subgroup,
and all $\Gamma_k$, $k=1,\dots , p$, are left adjacent
classes in $\Gamma $.

Suppose $\gamma _0\in \gamma\Gamma_k\cap \Gamma_m$,
$\gamma_0 = \gamma\beta_0$, $\beta_0\in \Gamma_k$.
We may choose $0\ne u\in V_{\beta_0}$, $0\ne v \in V_{\gamma_0}$.
Since $A$ is graded irreducible, there exists $a\in A_\gamma $,
$b\in A_{\gamma^{-1}}$
such that $au=v$, $bv=u$.

If we choose $a_j=0$ for $j\ne k$ and $a_k=\mathrm{id}_{V_{\gamma_k}}$
in \eqref{eq:Ae-Structure} then
\[
e_k = \sum\limits_{\gamma\in \Gamma_k} T_\gamma \in A_e.
\]
Consider
$c=ae_kb \in A_e$. Since $cv=ae_k(bv)=ae_ku=v$, $c\ne 0$.
Assume $\Gamma_m\not\subseteq \gamma\Gamma_k$ and choose any
$\beta \in \Gamma_m\setminus \gamma\Gamma_k$.
Then for every $w\in V_\beta$ we have
$cw = ae_k(bw)=0$ because $bw\in V_{\gamma^{-1}\beta}$,
$\gamma^{-1}\beta\notin \Gamma_k$, $e_k(bw)=0$.
The element $c\in A_e$ constructed annihilates all $V_\beta $
for $\beta \in \Gamma_m\setminus \gamma\Gamma_k$, but
$cv=v\ne 0$ for $v\in V_{\gamma_0}$, $\gamma_0\in \Gamma_m$.
This is a contradiction to \eqref{eq:Ae-Structure}.
\end{proof}

Consider an arbitrary  $a\in A_\alpha $, $\alpha \in \Gamma $.
Then
\[
a = \sum\limits_{\gamma\in \Gamma } a_\gamma ,
\quad a_\gamma = aT_{\gamma }\in \Hom(V_\gamma, V_{\alpha\gamma}).
\]
Here $a_\gamma \in \Hom (V_\gamma, V_{\alpha\gamma })$.
Denote $a_{\gamma_k}$ by $a_k$, $k=1,\dots, p$.
Note that if $\gamma \in \Gamma_k$, $\alpha\gamma\in \Gamma_m$
then
$e_m a e_k \in A_\alpha $, so $\sum\limits_{\gamma \in \Gamma_k} a_\gamma
 \in A_\alpha $ for every $k=1,\dots, p$.

If $\alpha\gamma_k \notin \Gamma_0$ then
$W_{\alpha,\gamma }=\{b\in \Hom(V_{\gamma}, V_{\alpha\gamma})\mid
\exists a\in A_\alpha : a_\gamma =b \}\ne 0$
for every $\gamma \in \Gamma_k$.
Since $W_{\alpha,\gamma }$ is a
$(\End V_{\alpha\gamma})$--$(\End V_\gamma)$-bimodule embedded into
$\Hom (V_{\gamma}, V_{\alpha\gamma})$, we have
$W_{\alpha,\gamma } = \Hom (V_{\gamma}, V_{\alpha\gamma})$.

Given $\gamma\in \Gamma_k$, $\alpha \in \Gamma $, $\alpha\Gamma_k=\Gamma_m$,
the map $\sigma_{\alpha, \gamma}: a_k\mapsto a_\gamma $, $a\in A_\alpha$,
is a linear bijection between
$\Hom (V_{\gamma_k}, V_{\alpha\gamma_k})$ and
$\Hom (V_{\gamma}, V_{\alpha\gamma})$.
Assume $\sigma_{\alpha,\gamma_k}=\mathrm{id}$ and
$\sigma_{e,\gamma}(x) = \iota_\gamma x \iota_\gamma^{-1}$,
since we have already established the structure of $A_e$.

Suppose
$a = \sum\limits_{\gamma\in \Gamma_k} a_k^{\sigma_{\alpha,\gamma}}\in A_\alpha$,
$b = \sum\limits_{\beta\in \Gamma_m} \iota_\beta b_m \iota_{\beta}^{-1}T_\beta
 \in A_e$,
and consider
\[
ba = \sum\limits_{\gamma\in \Gamma_k} \iota_{\alpha\gamma}b_m
 \iota_{\alpha\gamma} a_{k}^{\sigma_{\alpha,\gamma}}.
\]
Here $(ba)_k = \iota_{\alpha\gamma_k}b_m\iota_{\alpha\gamma_k}^{-1}a_k$,
$(ba)_\gamma =
\iota_{\alpha\gamma}b_m\iota_{\alpha\gamma}^{-1}a_k^{\sigma_{\alpha,\gamma}}$.
Similarly, if
\[
c=
\sum\limits_{\beta\in \Gamma_k}
\iota_\beta c_k \iota_{\beta}^{-1}T_\beta\in A_e
\]
then
$(ac)_k = a_kc_k$, $(ac)_\gamma = a_k^{\sigma_{\alpha, \gamma}}\iota_\gamma c_k
 \iota_\gamma^{-1}$.

Summarizing, if $a\in \Hom (V_{\gamma_k}, V_{\alpha\gamma_k})$ then
\begin{equation}\label{eq:ModLinearity}
(bac)^{\sigma_{\alpha,\gamma}}
= \iota_{\alpha\gamma}
\iota_{\alpha\gamma_k}^{-1}
 b \iota_{\alpha\gamma_k} \iota_{\alpha\gamma}^{-1}
  a^{\sigma_{\alpha,\gamma }} \iota_\gamma c\iota_\gamma^{-1}
\end{equation}
for all
$b\in \End V_{\alpha\gamma_k}$, $c\in \End V_{\gamma_k}$.
Define the linear transformation
$\sigma'_{\alpha,\gamma}$ of $\Hom(V_{\gamma_k},V_{\alpha\gamma_k})$
by the following rule:
\[
\sigma'_{\alpha,\gamma}: x\mapsto
\iota_{\alpha\gamma_k}\iota_{\alpha\gamma}^{-1}
    x^{\sigma_{\alpha,\gamma}} \iota_\gamma ,
\quad
x\in \Hom(V_{\gamma_k},V_{\alpha\gamma_k}).
\]
Then
$(bac)^{\sigma'_{\alpha,\gamma}} = b a^{\sigma'_{\alpha,\gamma}} c$
for all
$b\in \End V_{\alpha\gamma_k}$, $c\in \End V_{\gamma_k}$.
Hence, $x^{\sigma'_{\alpha, \gamma}}= \chi(\alpha, \gamma)x$,
where $\chi: \Gamma \times \Gamma \to \Bbbk $,
and we may assume $\chi(\alpha, \gamma)=0$ for
$\gamma \in \Gamma_0$ or $\alpha\gamma \in \Gamma_0$.
Also, $\chi(\alpha, \gamma_k)=1$, $\chi(e,\gamma)=1$ for
$\gamma \notin \Gamma_0$.
Finally,
$x^{\sigma_{\alpha,\gamma}}=
\chi(\alpha, \gamma)\iota_{\alpha\gamma}\iota_{\alpha\gamma_k}^{-1}
  x \iota_\gamma^{-1} $,
which implies \eqref{eq:A-Structure}.

It remains to derive \eqref{eq:ZCondition}.
Consider two elements $a\in A_\alpha $, $b\in A_\beta $, where
\[
a=\sum\limits_{\gamma\in \Gamma_k}
\chi(\alpha, \gamma) \iota_{\alpha\gamma}\iota_{\alpha\gamma_k}^{-1} a_k \iota_\gamma^{-1} T_\gamma,
\quad
b=\sum\limits_{\delta\in \Gamma_m}
\chi(\beta, \delta) \iota_{\beta\delta}\iota_{\beta\gamma_m}^{-1} b_m \iota_\delta^{-1} T_\delta.
\]
If $c=ab\ne 0$ then $\beta\Gamma_m=\Gamma_k$, $\alpha\Gamma_k\not\subseteq \Gamma_0$.
Straightforward computation shows
\[
c= \sum\limits_{\delta\in \Gamma_m} \chi(\alpha, \beta\delta)\chi(\beta, \delta)
 \iota_{\alpha\beta\delta}\iota^{-1}_{\alpha\beta\gamma_m}
 \big(\iota_{\alpha\beta\gamma_m}\iota_{\alpha\gamma_k}^{-1}a_k\iota_{\beta\gamma_m}^{-1} b_m \big) \iota_\delta^{-1} T_\delta.
\]
On the other hand,
\[
c=\sum\limits_{\delta\in \Gamma_m}
\chi(\alpha\beta, \delta) \iota_{\alpha\beta\delta}\iota_{\alpha\beta\gamma_m}^{-1} c_m \iota_\delta^{-1} T_\delta,
\]
so for $\delta=\gamma_m$ we have
\[
c_m =
\chi(\alpha, \beta\gamma_m)
\big(\iota_{\alpha\beta\gamma_m}\iota_{\alpha\gamma_k}^{-1}a_k\iota_{\beta\gamma_m}^{-1} b_m \big).
\]
Therefore,
$\chi(\alpha\beta, \delta)\chi(\alpha, \beta\gamma_m)
= \chi(\alpha, \beta\delta)\chi(\beta,\delta)$.
\end{proof}

\begin{cor}\label{cor:Irred->Simple}
A graded irreducible subalgebra of $\End V$ is simple.
\end{cor}

\begin{proof}
Suppose $A$ is a graded irreducible subalgebra of $\End V$ as described by
Theorem \ref{thm:A_Structure}, $n=n_1+\dots + n_p$.
The basis of this algebra over $\Bbbk $ consists of
\[
X_{(m,i),(k,j)}(\gamma ) =
 \sum\limits_{\beta\in \Gamma_k} \chi(g, \beta) \iota_{g\beta}\iota_{g\gamma_k}^{-1}
 E_{ij} \iota_\beta^{-1} T_\beta,
\quad
g =\gamma_m\gamma \gamma_k^{-1},
\]
where
$k,m=1,\dots, p$, $i=1,\dots, n_m$, $j=1,\dots, n_k$, $\gamma \in \Gamma_1$,
and $E_{ij}: V_{\gamma_k} \to V_{\gamma_m}$ are linear transformations corresponding
to matrix units in some fixed homogeneous basis of~$V$.

Consider the space of matrices $M_n(\Bbbk \Gamma_1)$ over the group algebra
$\Bbbk \Gamma_1$. Let us enumerate rows and columns of these matrices by
pairs $(k,j)$, $k=1,\dots, p$, $j=1,\dots, n_k$, instead of natural numbers
$1,\dots, n$.
Given a 2-cocycle $\chi \in Z^2(\Gamma_1,\Bbbk^*)$, denote by
$\Bbbk^\chi \Gamma_1$ the same space $\Bbbk\Gamma_1$ endowed with 
twisted multiplication $\gamma \cdot \beta = \chi(\gamma, \beta) \gamma\beta$,
$\gamma,\beta \in \Gamma_1$.
This graded algebra as well as matrix algebras $M_n(\Bbbk^\chi\Gamma_1)$
with natural grading are known to be simple, see, e.g.,
\cite{BZS}.

Let $\Phi $ stands for the linear bijection from
$M_n(\Bbbk\Gamma_1)$ to $A$ defined by
\begin{equation}\label{eq:GrIsomorphism}
\Phi : E_{(m,i),(k,j)}\gamma \mapsto \chi(\gamma_m, \gamma) X_{(m,i),(k,j)}(\gamma ),
\quad \gamma \in \Gamma_1.
\end{equation}
Straightforward computation shows that $\Phi $ is an isomorphism of
graded algebras
$M_n(\Bbbk^\chi\Gamma_1)$ and $A$. 
\end{proof}

In particular, we obtain one more proof of the structure
theorem for finite-dimensional graded algebras over an
algebraically closed field: Proposition~\ref{prop:CendEmbedding} and Corollary
\ref{cor:Irred->Simple} immediately imply

\begin{cor}[\cite{BZS,ON1982}]
A finite-dimensional simple graded associative algebra $A$
over an algebraically closed field
is isomorphic to $M_n(\Bbbk^\chi \Gamma_1)$,
where $\chi \in Z^2(\Gamma_1, \Bbbk^*)$.
\end{cor}

\section{Graded conformal algebras over
the affine line}

In this section, we consider graded conformal algebras over $G^0=\mathbb A^1$,
where $\mathbb A^1$ is the additive group of the base
field $\Bbbk $ which is assumed to be
algebraically closed and $\mathrm{char}\,\Bbbk =0$.
In this case,
$H^0=\Bbbk [T]$,
$\Aut G^0=\Bbbk ^*$,  and a homomorphism
$\sigma:\Gamma \to \Aut G^0$ maps $\gamma \in \Gamma $ to
a nonzero scalar $\sigma(\gamma )\in \Bbbk^* $, it acts on
$\mathbb A^1$ by multiplication.

\begin{rem}\label{rem:TrivialCocycle}
If $\Gamma $ is a finite group then
for every $\sigma:\Gamma \to \Bbbk^*$ all cocycles
in $Z^2(\Gamma , \Bbbk , \sigma)$ are cohomological to
the zero cocycle (see \cite[III.10.2]{Brown}).
Therefore, to classify
finite simple associative conformal algebras over a linear algebraic
group $G$ such that $G^0\simeq \mathbb A^1$ it is enough to consider only
$(\Gamma, \sigma, 0)$-graded conformal algebras.

However, if $\Gamma $ is infinite then there may exist nontrivial cocycles.
This case does not correspond to a linear algebraic group, but
we still can consider graded conformal algebras in these settings.
\end{rem}

By Proposition \ref{prop:CendEmbedding},
to classify simple finite $(\Gamma , \sigma,\varphi)$-graded associative
conformal algebras it is enough to describe
simple finite graded irreducible subalgebras of the
$(\Gamma ,\sigma,\varphi)$-graded conformal algebra $\Cend^\Gamma M$,
where $M=\bigoplus\limits_{\gamma\in \Gamma  }M_\gamma $
is a finitely generated torsion-free  $H^0$-module.
Such a module over $\Bbbk[T]$ is free.

Let us state in details the structure of $\Cend^\Gamma M$ in the graded
case. Suppose $N$ is the rank of $M$ over $H^0$, and let
  $e_1,\dots , e_N$ be a homogeneous basis of $M$.
For each $i=1,\dots, N$ there exists $\alpha_i\in \Gamma $
such that $e_i\in M_{\alpha_i}$.
An arbitrary element $a\in \Cend^\Gamma M$ is completely defined
by the family of regular functions
$(a\oo{\lambda} e_i)=\sum\limits_{s\ge 0} \sum\limits_{j=1}^N
(- \lambda)^s a^s_{ij}(T)e_j$,
$i=1,\dots, N$,
where $a_{ij}^s(T)\in H^0$.
Hence, $a\in \Cend^\Gamma M$ can be identified with the matrix
$[a]=\big (a_{ij}(T,x)\big )\in M_N(\Bbbk[T,x])$,
where
\[
a_{ij}(T,x)=\sum\limits_{s\ge 0} (\sigma(\alpha_j\alpha_i^{-1})T
 +\varphi(\alpha_i\alpha_j^{-1}, \alpha_j\alpha_i^{-1}))^sa_{ij}^s(x).
\]
The $H^0$-module $\Cend^\Gamma M$ carries the natural $\Gamma $-grading given
by
\begin{equation}\label{eq:NaturalGrading}
(\Cend^\Gamma M)_\gamma = \sum\limits_{i,j: \gamma =\alpha_i\alpha_j^{-1}}
 \Bbbk(T,x)E_{ij},
\end{equation}
where $E_{ij}$ are the matrix units.

The operations $(\cdot\oo{\lambda }\cdot )$, $\lambda \in \mathbb A^1$,
on matrices over $\Bbbk[T,x]$
may be explicitly computed by \eqref{eq:CendOperations}:
If $a=f(T,x)E_{ij}$, $b=g(T,x)E_{jk}$ then
\begin{multline}\label{eq:CendOperationsMatrix}
(a\oo{\lambda } b)
=
f(-\sigma(\alpha)(\lambda+\varphi(\alpha,\alpha^{-1})), x)  \\
\times
g(T+\sigma(\alpha\beta)\lambda
+\varphi(\alpha^{-1},\alpha\beta),
x +\sigma(\alpha_i)\lambda
+\varphi(\alpha^{-1},\alpha_i))E_{ik},
\end{multline}
where $\alpha = \alpha_i\alpha_j^{-1}$,
$\beta = \alpha_j\alpha_k^{-1}$.

Let $V=\bigoplus\limits_{\gamma\in \Gamma } V_\gamma $
be an $N$-dimensional $\Gamma $-graded linear space,
$\dim V_\gamma = n_\gamma $,
 and let
$M=H^0\otimes V$ be the free module with the inherited  grading,
$M_\gamma =H^0\otimes V_\gamma $.

Given a homogeneous basis $e=(e_1,\dots, e_N)$ of $V$, consider
$\Cend^\Gamma_N
 =\{[a]_e\mid a\in \Cend^\Gamma M\}
\simeq
M_N(\Bbbk [T,x])$
as a $(\Gamma,\sigma,\varphi)$-graded
conformal algebra in the matrix form
(the grading is completely defined by $\alpha_1,\dots, \alpha_N\in \Gamma$,
as above).

Suppose
there is another homogeneous $H^0$-basis $f=(f_1,\dots, f_N)$ in $M$ such that
$(f_1,\dots, f_N)=(e_1,\dots, e_N)Q(T)$,
$Q$ is an invertible matrix in
$M_N(\Bbbk[T])$ which can be presented in a block-diagonal form
$Q(T)=\bigoplus\limits_{\gamma \in \Gamma\setminus \Gamma_0}Q_\gamma(T)$,
$Q_\gamma,Q_\gamma^{-1} \in M_{n_\gamma}(\Bbbk[T])$.
This is straightforward to compute that for a given
$a\in (\Cend^\Gamma M)$
its matrix in the new basis can be found as follows:
$[a]_f = Q^{-1}(x)[a]_e Q^\Gamma(T,x)$,
where
$Q^\Gamma(T,x) = \bigoplus\limits_{\gamma\in \Gamma\setminus\Gamma_0}
  Q_\gamma(x-\sigma(\gamma)T)$.

\begin{thm}\label{thm:FiniteIrreducible}
A finite graded irreducible conformal subalgebra
$C\subseteq \Cend^\Gamma _N$
is isomorphic to the current algebra $\Cur A$
over a graded irreducible subalgebra $A\subseteq \End V$.
\end{thm}

\begin{proof}
Denote by $\Gamma _0$ the set of all $\gamma \in \Gamma $
such that $V_\gamma =0$.
For each $\gamma \in \Gamma\setminus \Gamma_0 $
one may consider $M_\gamma $ as an ordinary $H^0$-module
and define the (non-graded) conformal algebra $\Cend M_\gamma $.
There exists a natural map
\begin{equation}\label{eq:Projection}
\pi_\gamma : C_e\to \Cend M_\gamma ,
\quad
 \pi_\gamma\bigg(\sum\limits_{i,j}a_{ij}(T,x)E_{ij} \bigg) =
 \sum\limits_{i,j:\alpha_i=\alpha_j=\gamma } a_{ij}(T, \sigma(\gamma)x)E_{ij}.
\end{equation}
It follows from \eqref{eq:CendOperationsMatrix} that
$\pi_\gamma $ is a homomorphism of conformal algebras.
Moreover,
if $\mu_\gamma: M_\gamma \to M_\gamma $
is the map defined by $h(T)\otimes e \mapsto h(\sigma(\gamma)T)\otimes e$
then
\[
\mu_{\gamma}^{-1} (\pi_\gamma(a))(\lambda ) \mu_\gamma = a(\lambda)|_{M_\gamma }
\]
for all $a\in C_e$. Therefore, irreducibility of $C$
implies $\pi_\gamma (C_e)$ is a finite irreducible
subalgebra of $\Cend M_\gamma $.
By \cite{BKL2003}, the latter coincides with
$\Cur\End \Bbbk^{n_\gamma } $
if we make an appropriate choice of an $H^0$-basis in $M_\gamma $.

The latter means (see, e.g., \cite{Kol2006Adv}) that
there exists an invertible matrix $Q_\gamma \in M_{n_\gamma}(\Bbbk[x])$
such that
\[
Q_\gamma ^{-1}(x) \pi_\gamma (C_e) Q_\gamma (x-T) = M_{n_\gamma }(\Bbbk[T]).
\]
Let us choose a new $H^0$-basis in $M$ with transition matrix
\[
Q(T) = \bigoplus\limits_{\gamma \in \Gamma\setminus \Gamma_0}
Q_\gamma(\sigma(\gamma)^{-1}T).
\]
In this basis,  matrices of all elements from $ C_e$ do not depend
on~$x$, i.e., $\pi_\gamma (C_e) = \Cur\End \Bbbk^{n_\gamma }$.

By Proposition \ref{prop:IrreducibleSemisimple},
$C_e$ is a semisimple finite conformal algebra.
As shown in \cite{DK1998} (see also \cite{Kol2006Adv}),
$C_e$ is isomorphic to the direct sum
of current conformal algebras over matrix algebras, i.e.,
\[
C_e=I_1\oplus \dots \oplus I_p, \quad I_k\simeq \Cur \End \Bbbk^{n_k}.
\]
Denote by $\pi_k$ the isomorphism $\Cur \End \Bbbk^{n_k} \to I_k$.

All conformal algebras $\Cur \End \Bbbk^{n_k }$ are simple, so we may
split $\Gamma \setminus \Gamma _0$ into adjacent classes
\[
\Gamma_k
=\{\gamma\in \Gamma \mid \pi_\gamma (I_k)= \Cur\End \Bbbk^{n_\gamma } \},
\quad
k=1,\dots, p.
\]
In particular, $n_\gamma =n_k$ for every $\gamma \in \Gamma _k$.

For every $k=1,\dots, p$ and for every $\gamma \in \Gamma_k$
we may consider the automorphism $\theta_\gamma = \pi_\gamma \pi_k :
\Cur\End \Bbbk^{n_k} \to \Cur\End \Bbbk^{n_k} $

\begin{lem}\label{lem:AutomorphismsOfCurrent}
Let $\theta  $
be an automorphism of the conformal algebra
$\Cur \End \Bbbk^{n}\simeq M_n(\Bbbk[T])$.
Then $\theta(a) = Q^{-1} a Q $, $a\in \Cur \End \Bbbk^n$,
for some invertible matrix $Q\in M_n(\Bbbk)$.
\end{lem}

In particular, $\theta(E_n)=E_n$.

\begin{proof}
Let us extend $\theta $ to a map $\bar\theta:\Cend_n \to \Cend _n$, where
$\Cend_{n}=\Cend (H\otimes \Bbbk^{n})\simeq M_n(\Bbbk[T,x])$
by the following rule:
\[
\bar\theta(x^ka )=x^k\theta (a), \quad a\in \Cur \End \Bbbk^n,\
k\ge 0.
\]
It is easy to see that the map $\bar\theta$ obtained is an automorphism
of the conformal algebra $\Cend_n$.

All automorphisms of $\Cend_{n}$ were described in \cite{BKL2003},
see also \cite{Kol2006Adv}. There exist a scalar $\alpha \in \Bbbk $
and an invertible matrix $Q(t)\in M_n(\Bbbk[t])$ such that
for a given $a=a(T,x)\in \Cend_n$ we have
$\bar\theta(a(T,x)) = Q^{-1}(x) a(T,x+\alpha ) Q(x-T)$.
Since $\bar\theta $ preserves $\Cur\End \Bbbk^n$
 whose elements do not depend of~$x$, the components of $Q$ do not depend
on $x$ as well.
\end{proof}

Therefore, for every $\gamma \in \Gamma_k$, $k=1,\dots, p$,
$\theta(E_{n_k}) = \pi_\gamma (\varepsilon_k) = E_{n_k}$,
where $\varepsilon_k = \pi_k(E_{n_k}) \in I_k$.
Since $\pi_\gamma$ is given by \eqref{eq:Projection},
$\varepsilon_k = E_{N_k}$, where $N_k=|\Gamma_k|n_k$.
Hence, $C_e$ contains the identity matrix
$E_N=\varepsilon_1+\dots +\varepsilon_p $.

It remains to describe
finite graded irreducible subalgebras
 $C\subseteq \Cend^\Gamma_N$
such that $\pi_\gamma(C_e) = \Cur\End \Bbbk^{n_\gamma }$
for every $\gamma \in \Gamma\setminus \Gamma_0$
and $E=E_N\in C_e$.

Suppose $a=a(T,x)$ is a nonzero element of $C_\alpha $ for some
$\alpha \in \Gamma $.
Then
\begin{equation}\label{eq:T-decomp}
a = \sum\limits_{s\ge 0}(\sigma(\alpha^{-1})T +\varphi(\alpha, \alpha^{-1}))
a^s(x),
\quad a^s(x) =
\sum\limits_{i,j: \gamma_i\gamma_j^{-1}=\alpha}a_{ij}^s(x) E_{ij},
\end{equation}
where $a_{ij}^s(x) \in \Bbbk[x]$.

Since
$a\oo{\lambda } E \in C_\alpha$ for every $\lambda \in \mathbb A^1$,
we may conclude $a^s(x) \in C_\alpha $ for all $s\ge 0$.
Therefore,
$C=H^0\otimes A$, where $A$ is a graded subspace of
$M_N(\Bbbk[x])$ with respect to the same natural grading.

Assume $a=a^0(x)$, and choose $i,j\in \{1,\dots, N\}$
such that $a_{ij}(x)=a_{ij}^0\ne 0$.
Denote $\gamma =\gamma_j$, then $\gamma_i=\alpha\gamma $.

Since both $M_\gamma, M_{\alpha\gamma}\ne 0$, there exists
$0\ne b\in C_{\alpha^{-1}}$.
We may assume $b=b(x)$, and choose $k,l\in \{1,\dots, N\}$
such that $\gamma_l=\gamma $, $\gamma_k=\alpha\gamma$, $b_{lk}(x)\ne 0$.

For every $k\in \{1,\dots, N\} $ such that
$\gamma_k=\alpha\gamma $
there exists $\varepsilon_{ki}\in C_e$ such that
$\pi_{\alpha\gamma}(\varepsilon_{ki})=E_{ki}
 \in \Cur M_{n_{\alpha\gamma}}(\Bbbk)$
(assuming the numeration of rows and columns
in this ``small'' matrix is the same as in the initial
$(N\times N)$-matrix $a$).

Then
$c=b \oo{0} (\varepsilon_{ki} \oo{0} a)\in C_e$,
but \eqref{eq:CendOperations} implies
$c(0) e_j = b(0)\varepsilon_{ki}(0)a(\lambda )e_j$.
The latter is equal to
$b_{lk}(T)a_{ij}(T+\varphi(\alpha, \gamma))e_l\in M_\gamma $,
i.e., if either of $a_{ij}$ or $b_{kl}$  depend on~$x$
then some nonzero coefficients of $c\in C_e$ also depend on~$x$.
This is impossible since $C_e$ is a direct sum of current algebras,
$\pi_\gamma (C_e)=\Cur\End \Bbbk^{n_\gamma}
=  M_{n_\gamma}(\Bbbk[T])$.

Therefore, if we choose a basis of $M$ over $H^0$ as described above
then all coefficients of all matrices in $A$ do not depend
on $x$. Hence, $A$ is a graded subalgebra of
$M_N(\Bbbk )\simeq \End V$. It is easy to see that
the irreducibility of $C$ implies $A$ to be graded irreducible
subalgebra of $\End V$.
\end{proof}

Corollary \ref{cor:Irred->Simple} implies

\begin{cor}\label{cor:IrredSimple}
A finite graded irreducible conformal subalgebra
$C\subseteq \Cend^\Gamma _N$ is simple.
\end{cor}

\begin{cor}\label{cor:SimpleGraded}
A finite simple $(\Gamma,\sigma,\varphi)$-graded
associative conformal algebra is isomorphic to
the current  graded conformal algebra $\Cur A$
over a simple finite-dimen\-si\-onal graded associative
algebra~$A$.
\end{cor}

\begin{proof}
By Proposition \ref{prop:CendEmbedding} and Theorem \ref{thm:FiniteIrreducible}
we have $C\simeq \Cur A$ for a finite-dimensional graded
associative algebra $A$. The graded algebra $A$ has to be simple
since a nonzero proper graded ideal $I\triangleleft A$
produces similar ideal $\Cur I$ in~$C$.
\end{proof}

Now, consider semisimple graded associative conformal algebras.
We need a statement that can be proved in general settings.

Let $G^0$ be a connected linear algebraic group, $\Gamma $ be a group,
and let $M$ be a finitely generated torsion-free
$\Gamma $-graded $H^0$-module.
As before, fix a homomorphism $\sigma : \Gamma \to \Aut G^0$
and a cocycle $\varphi\in Z^2(\Gamma, G^0, \sigma)$.

\begin{lem}\label{lem:MinimalFaithful}
Let $C$ be a semisimple $(\Gamma, \sigma, \varphi)$-graded
associative conformal algebra over $G^0$,
and let $M$ be a faithful graded conformal $C$-module
which is torsion-free and finitely
generated as an $H^0$-module.
Then the set of all nonzero faithful graded $C$-submodules of
$M$ contains a minimal element.
\end{lem}

\begin{proof}
Suppose
\[
M=M_0\supseteq M_1 \supseteq M_2 \supseteq \dots
\]
is a descending chain of faithful graded $C$-submodules.
Since $M$ is Noetherian, all $M_k$ are finitely generated over $H^0$.

Consider the field of fractions $Q(H^0)$ and the linear space
$Q(M)=Q(H^0)\otimes _H M$. Since $M$ is torsion-free, $M$ is a $\Bbbk $-subspace
of $Q(M)$ and all $M_k$ span finite-dimensional $Q(H^0)$-subspaces of $Q(M)$.
Therefore, there exists $N\ge 0$ such that for every $n>N$ there exists
$0\ne h\in H^0$ such that $hM_N \subseteq M_n$.

By Lemma \ref{lem:TorsionAnnihilates}, $(C\oo{\lambda } M_N)\subseteq M_n$
for all $\lambda \in G^0$. The $H^0$-submodule $CM_N=H^{0}\{(a\oo{\lambda } u)
\mid a\in C,\, u\in M_N,\, \lambda \in G^0 \}$
is a nonzero faithful graded conformal $C$-submodule of $M$
which is a lower bound of all $M_k$ in the initial chain.
The Zorn Lemma implies the claim.
\end{proof}

\begin{thm}
If every finite irreducible $(\Gamma, \sigma , \varphi)$-graded subalgebra
of $\Cend ^\Gamma M$ is simple then every finite
semisimple $(\Gamma, \sigma , \varphi)$-graded subalgebra
of $\Cend^\Gamma M$ is a direct sum of simple ones.
\end{thm}

\begin{proof}
Suppose $C\subseteq \Cend^\Gamma M$ is a semisimple $(\Gamma, \sigma , \varphi)$-graded subalgebra.
It follows from Lemma \ref{lem:MinimalFaithful} that
there exists a minimal faithful graded conformal $C$-submodule $M_0\subseteq M$,
so $C\subseteq \Cend^\Gamma M_0$.

Let $U$ be a maximal graded conformal $C$-submodule of $M_0$.
If $U=0$ then $M_0$ is irreducible, so $C$ is simple.
If $U\ne 0$ then $I=\{a\in C\mid aU=0\} $ is a nonzero graded ideal of~$C$.
Moreover, if
$U'$ stands for $\{v\in M_0\mid Iv =0\}\subset M_0$ then $U=U'$
from the maximality of~$U$.

Denote $J=\{a\in C\mid Ia=0 \}$. This is a graded ideal in $C$, and $I\cap J=0$.
Note that $JM\subseteq U' = U$. Conversely, if
$aM_0\subseteq U$ then $IaM_0\subseteq IU=0$, so $M/U$ is
a faithful graded conformal $C/J$-module. This module is irreducible
by the choice of~$U$.

Hence, by the assumption of the Theorem, $C/J$ is simple, so $I\simeq C/J$, and
$C=I\oplus J$. It is obvious that $J$ is semisimple, and we may apply the same
procedure to~$J$.
Noetherianity of $C$ implies the process to stop on a finite number of steps,
so $C = I_1\oplus \dots \oplus I_p$, where $I_k$ are simple finite
graded conformal algebras described by Corollary~\ref{cor:SimpleGraded}.
\end{proof}

\begin{cor}
A finite semisimple $(\Gamma, \sigma, \varphi)$-graded associative
conformal algebra over the affine line is a direct sum of simple ones.
\end{cor}

\begin{proof}
It is enough to note that a finite semisimple graded conformal algebra $C$
embeds into $\Cend^\Gamma C$, and then apply Corollary \ref{cor:IrredSimple}.
\end{proof}

\section*{Acknowledgements}
The work is supported by the Russian Foundation for Basic Research (project 09--01--00157) 
and the Federal Target Grant ``Scientific and
educational staff of innovation Russia'' for 2009--2013
(contracts 02.740.11.5191, 02.740.11.0429, 14.740.11.0346).
The author is grateful to Valeri\u i Churkin for helpful
discussions.

\end{document}